\documentclass[12pt]{amsart}

\theoremstyle{plain}
\newtheorem{theorem}{Theorem}[section]

\newtheorem{conjecture}[theorem]{Conjecture}
\newtheorem{corollary}[theorem]{Corollary}

\newtheorem{lemma}[theorem]{Lemma}

\newtheorem{proposition}[theorem]{Proposition}

\newtheorem{problem}[theorem]{Problem}

\theoremstyle{definition}

\numberwithin{equation}{section}

\usepackage{fullpage}
\usepackage{amssymb}
\usepackage{tikz}
\def\ldiv{\backslash}

\def\Z{\mathbb Z}

\def\W{\mathcal{W}}

\def\mlt#1{\mathrm{Mlt}(#1)}
\def\inn#1{\mathrm{Inn}(#1)}

\def\aut#1{\mathrm{Aut}(#1)}
\def\tmlt#1{\mathrm{TMlt}(#1)}
\def\tinn#1{\mathrm{TInn}(#1)}

\def\abext#1#2#3{#1\,{:}_{#3}\,#2}
\def\id{\mathrm{id}}
\def\restr#1{{\restriction}_{#1}}

\begin{document}

\title{Abelian extensions and solvable loops}

\author{David Stanovsk\'y}

\author{Petr Vojt\v{e}chovsk\'y}

\address[Stanovsk\'y]{Department of Algebra, Faculty of Mathematics and Physics, Charles University, Sokolovsk\'a 83, 18675 Praha 8, Czech Republic}
\address[Vojt\v{e}chovsk\'y]{Department of Mathematics, University of Denver, 2280 S Vine St, Denver, Colorado 80208, U.S.A.}
\email[Stanovsk\'y]{stanovsk@karlin.mff.cuni.cz}
\email[Vojt\v{e}chovsk\'y]{petr@math.du.edu}

\begin{abstract}
Based on the recent development of commutator theory for loops, we provide both syntactic and semantic characterization of abelian normal subloops. We highlight the analogies between well known central extensions and central nilpotence on one hand, and abelian extensions and congruence solvability on the other hand. In particular, we show that a loop is congruence solvable (that is, an iterated abelian extension of commutative groups) if and only if it is not Boolean complete, reaffirming the connection between computational complexity and solvability. Finally, we briefly discuss relations between nilpotence and solvability for loops and the associated multiplication groups and inner mapping groups.
\end{abstract}

\keywords{Commutator theory for loops, abelian normal subloop, solvable loop, nilpotent loop, abelian extension, central extension.}

\subjclass[2010]{20N05, 08B10}

\thanks{Research partially supported by GA\v CR grant 13-01832S (David Stanovsk\'y) and by Simons Foundation Collaboration Grant 210176 (Petr Vojt\v{e}chovsk\'y).}

\maketitle

\section{Introduction}

Building on the pioneering work of Smith \cite{S} and others, in 1987 Freese and McKenzie developed commutator theory for congruence modular varieties \cite{FM}. Since the variety of loops is congruence modular, the Freese-McKenzie theory can be routinely transferred to loops. However, it is not routine to describe the commutator of two loop congruences (equivalently, two normal subloops) efficiently and with regard to the permutation groups typically associated with loops. This has been recently done in \cite{SV}, laying foundations for commutator theory in loops. The main result of \cite{SV} is restated here as Theorem \ref{Th:mainSV}.

Continuing the program of \cite{SV}, we take a closer look at central and abelian subloops, two notions derived from the Freese-McKenzie commutator theory. The main result of this paper, Theorem \ref{Th:main_abelian}, is a syntactic and semantic characterization of abelian subloops, analogous to the well-known characterization of central subloops, reviewed in Theorem \ref{Th:main_central}.

The semantic characterization of abelian normal subloops leads to the notion of abelian extensions, a generalization of central extensions. Unlike for groups, abelian extensions of loops are manageable objects because the cocycle condition that ensures that the construction yields a loop is simple and of combinatorial character.

Following universal algebra, nilpotent loops should be precisely the loops obtained by iterated central extensions. It turns out that this notion of nilpotence coincides with Bruck's central nilpotence (finite upper central series), as proved for instance in \cite{SV}.

Analogously, solvable loops should be precisely the loops obtained by iterated abelian extensions. It turns out that this notion of solvability is strictly stronger than Bruck's solvability (subnormal series with commutative groups as factors), as noted already in \cite{FM}. To avoid confusion, we call the former \emph{congruence solvability} and the latter \emph{classical solvability}. The two notions coincide in groups, where we will safely use the term \emph{solvability}.

There is an interesting connection between solvability and \emph{Boolean completeness}, the ability of an algebraic structure to express arbitrary Boolean functions via terms involving constants and variables. It is well known, cf. \cite{MaRh, Straubing}, that a group is Boolean complete if and only if it is not solvable. In \cite{LMT,MTLBD}, the authors set out to generalize this result to loops (and quasigroups). To that effect, they defined a certain extension of loops, called \emph{affine quasidirect product}, and let \emph{polyabelian} loops to be precisely the loops obtained by iterated affine quasidirect products of commutative groups. They proved that loops are Boolean complete if and only if they are not polyabelian. Consequently, they argued that polyabelianess should be the correct notion of solvability for loops, and they proved that it lies strictly between nilpotency and classical solvability.

As a byproduct of our investigation of the Freeze-McKenzie commutator theory in loops, we show that abelian extensions of loops are precisely the affine quasidirect products of \cite{LMT,MTLBD} and, consequently, that a loop is Boolean complete if and only if it is not congruence solvable. This generalizes the connection between solvability and Boolean completeness from groups, and we believe that it will lead to new applications of commutator theory in computational complexity.

The paper is organized as follows. Section \ref{Sec:comm} gives a brief introduction to commutator theory for loops. Abelian extensions are developed in Section \ref{Sec:extensions}. The main result, Theorem \ref{Th:main_abelian}, is proved in Section \ref{Sec:main}. In Section \ref{Sec:sn} we discuss dependencies between the two notions of solvability, central nilpotence, and supernilpotence, as applied to loops, their multiplication groups and their inner mapping groups.

\section{Commutator theory for loops}\label{Sec:comm}

In this section we present a condensed introduction to commutator theory for loops, compiled from \cite{SV}, where all proofs, examples and additional details  can be found. See \cite{B,P} for an introduction to loop theory, \cite{Berg} for universal algebra, and \cite{MS} for general commutator theory.

For any groupoid $(Q,\cdot)$ and $x\in Q$ let $L_x:Q\to Q$, $R_x:Q\to Q$ be the translations
\begin{displaymath}
    L_x(y) = xy,\quad R_x(y) = yx.
\end{displaymath}
A groupoid $(Q,\cdot)$ is a \emph{quasigroup} if the translations $L_x$, $R_x$ are bijections of $Q$ for every $x\in Q$. Every quasigroup can be equipped with the division operations
\begin{displaymath}
    x\ldiv y = L_x^{-1}(y),\quad y/x = R_x^{-1}(y).
\end{displaymath}
A quasigroup $Q$ is a \emph{loop} if it possesses a neutral element $1$, that is, an element satisfying $1\cdot x = x\cdot 1 = x$ for every $x\in Q$. From now on, $Q$ will always denote a loop.

The \emph{multiplication group} and the \emph{inner mapping group} of $Q$ are defined as
\begin{displaymath}
    \mlt{Q} = \langle L_x,R_x:x\in Q\rangle,\quad \inn{Q} = \{\varphi\in\mlt{Q}:\varphi(1)=1\},
\end{displaymath}
respectively. It proved useful (although we do not know if it was necessary) in \cite{SV} to work not just with the multiplication and inner mapping groups, but also with the larger \emph{total multiplication group} and \emph{total inner mapping group}, defined by
\begin{displaymath}
    \tmlt{Q} = \langle L_x,R_x,M_x:x\in Q\rangle,\quad \tinn{Q} = \{\varphi\in\tmlt{Q}:\varphi(1)=1\},
\end{displaymath}
respectively, where $M_x:Q\to Q$ is given by
\begin{displaymath}
    M_x(y) = y\ldiv x.
\end{displaymath}
Note that $M_x^{-1}(y) = x/y$. We refer to elements of $\inn{Q}$ (resp. $\tinn{Q}$) as \emph{inner mappings} (resp. \emph{tot-inner mappings}).

Consider the mappings
\begin{align*}
    T_x &= R_x^{-1}L_x,\\
    U_x &= R_x^{-1}M_x,\\
    L_{x,y} &= L_{L_x(y)}^{-1}L_xL_y = L_{xy}^{-1}L_xL_y,\\
    R_{x,y} &= R_{R_x(y)}^{-1}R_xR_y = R_{yx}^{-1}R_xR_y,\\
    M_{x,y} &= M_{M_x(y)}^{-1}M_xM_y = M_{y\ldiv x}^{-1}M_x M_y.
\end{align*}
It is well known that
\begin{displaymath}
    \inn{Q} = \langle T_x,L_{x,y},R_{x,y}:x,y\in Q\rangle,
\end{displaymath}
and one can prove in a similar way that
\begin{displaymath}
    \tinn{Q} = \langle T_x,U_x,L_{x,y},R_{x,y},M_{x,y}:x,y\in Q\rangle.
\end{displaymath}
This and other generating sets of $\tinn{Q}$ can be found in \cite{SV}.

A subloop $A$ of $Q$ is \emph{normal} in $Q$ if it is a kernel of a homomorphism. This happens if and only if $\varphi(A)=A$ for every $\varphi\in\inn{Q}$, or equivalently, if and only if $\varphi(A)=A$ for every $\varphi\in\tinn{Q}$.

A \emph{word} $W$ is a formal product of letters $L_{t(\bar x)}$, $R_{t(\bar x)}$ and their inverses, where $t(\bar x)=t(x_1,\dots,x_n)$ is a loop term. For instance, $W = L_xR_y^{-1}L_{x/y}$ is a word. Upon substituting elements $u_i$ of a particular loop $Q$ for the variables $x_i$ in a word $W$ and upon interpreting $L_{t(\bar u)}$, $R_{t(\bar u)}$ as translations of $Q$, we obtain $W_{\bar u}$, an element of $\mlt{Q}$. If $W_{\bar u}(1)=1$ for every loop $Q$ with neutral element $1$ and every assignment of elements $u_i\in Q$, we say that $W$ is an \emph{inner word}. For instance, $R_{xy}^{-1}L_xR_y$ is an inner word, while $R_xL_x$ is not. The concept of \emph{tot-inner word} is defined similarly, allowing $M_{t(\bar x)}$ as generating letters.

A detailed and well-motivated definition of the commutator of two loop congruences can be found in \cite{SV}. The following result from \cite{SV} describes the commutator in purely loop theoretical fashion.

\begin{theorem}\label{Th:mainSV}
Let $\mathcal W$ be a set of tot-inner words such that for every loop $Q$ we have $\tinn{Q} = \langle W_{\bar u}:W\in\mathcal W,\,u_i\in Q\rangle$.
Let $Q$ be a loop and $A$, $B$ two normal subloops of $Q$. The commutator $[A,B]_Q$ is the smallest normal subloop of $Q$ containing the set
$$\{W_{\bar u}(a)/W_{\bar v}(a):\ W\in\mathcal W,\,a\in A,\,u_i,v_i\in Q,\,u_i/v_i\in B\}.$$
\end{theorem}

For instance, $\mathcal W = \{T_x,U_x,L_{x,y},R_{x,y},M_{x,y}\}$ is a suitable set of tot-inner words for Theorem \ref{Th:mainSV}.

\medskip

A normal subloop $A$ of a loop $Q$ is called
\begin{itemize}
    \item \emph{central in $Q$} if $[A,Q]_Q=1$,
    \item \emph{abelian in $Q$} if $[A,A]_Q=1$.
\end{itemize}
Obviously, every normal subloop that is central in $Q$ is also abelian in $Q$.

The \emph{center} of $Q$ is the normal subloop $Z(Q)$ of $Q$ consisting of all elements that commute and associate with all elements of $Q$. The universal-algebraic notion of center congruence specialized to loops corresponds to the subloop $Z(Q)$, and it is easy to see that a normal subloop is central in $Q$ if and only if it is a subloop of $Z(Q)$.

A normal subgroup of a group $G$ is abelian in $G$ if and only if it is a commutative group. Not so in loops. If a normal subloop of a loop $Q$ is abelian in $Q$, then it is a commutative group, but there are numerous examples of a loop $Q$ with a normal subloop that is a commutative group that is not abelian in $Q$.

A loop $Q$ is said to be \emph{classically solvable} if there is a series $1=A_0\le A_1\le\cdots \le A_n=Q$ of subloops of $Q$ such that, for every $i$, $A_i$ is normal in $A_{i+1}$ and the factor $A_{i+1}/A_i$ is a commutative group. A loop $Q$ is said to be \emph{congruence solvable} (resp. \emph{nilpotent}) if there is a series $1=A_0\le A_1\le\cdots \le A_n=Q$ of normal subloops of $Q$ such that every factor $A_{i+1}/A_i$ is abelian in $Q/A_i$ (resp. central in $Q/A_i$).

It follows from the remarks above that our definition of nilpotence is equivalent to the classical concept of central nilpotence in loop theory. A group is classically solvable if and only if it is congruence solvable, and we thus speak of \emph{solvable} groups without any danger of confusion. It was noticed already in \cite{FM} that for loops congruence solvability is strictly stronger than classical solvability.

\medskip

However one decides to define the commutator of two elements and the associator of three elements in loop theory, it should be a quantity that vanishes when the elements commute or associate, respectively. It turns out that certain commutators and associators are more suitable than others for calculations, depending on the context. One of the observations of \cite{SV} is that a useful class of commutators and associators is obtained if these are based on deviations of (tot-)inner mappings from the identity mapping. For the purposes of this paper, we therefore define the \emph{commutator} by
\begin{displaymath}
    [y,x] = T_y(x)/x = ((yx)/y)/x,
\end{displaymath}
and the \emph{associator} by
\begin{displaymath}
    [x,y,z] = R_{z,y}(x)/x = (((xy)z)/(yz))/x.
\end{displaymath}
(The transposition of $x$ and $y$ in $[y,x]$ is employed so that the commutator specializes to one of the two usual commutators of group theory, namely, $[y,x] = yxy^{-1}x^{-1}$.) We will use the following observation frequently in the proof of Theorem \ref{Th:main_abelian}.

\begin{lemma}\label{Lm:useful}
Let $A$ be a subloop of a loop $Q$, and suppose that $[a,b,x]=1$ for all $a$, $b\in A$, $x\in Q$. Then, for every $a\in A$ and $u$, $v\in Q$,
\begin{itemize}
	\item[(i)] if $u/v\in A$, then $a(u/v) = (au)/v$;
	\item[(ii)] if $u\ldiv v\in A$, then $(u\ldiv v)a = u\ldiv (va)$.
\end{itemize}
\end{lemma}
\begin{proof}
(i) Since $u/v\in A$, we have $[a,u/v,v]=1$, hence $(a(u/v))v = a((u/v)v) = au$, and thus $a(u/v) = (au)/v$. Part (ii) is proved dually.
\end{proof}

\section{Abelian and central extensions}\label{Sec:extensions}

Let $A=(A,+,-,0)$ be a commutative group and $F=(F,\cdot,/,\ldiv)$ a quasigroup. A triple $\Gamma=(\varphi,\psi,\theta)$ is called a \emph{cocycle} if $\varphi$, $\psi$ are mappings $F\times F\to\aut A$ and $\theta$ is a mapping $F\times F\to A$. The values of the mappings will be denoted shortly by $\varphi(x,y)=\varphi_{x,y}$, $\psi(x,y)=\psi_{x,y}$ and $\theta(x,y)=\theta_{x,y}$.

Given a cocycle $\Gamma = (\varphi,\psi,\theta)$, define a multiplication on the set $A\times F$ by
\begin{equation}\label{Eq:abext}
    (a,x)\cdot(b,y)=(\varphi_{x,y}(a)+\psi_{x,y}(b)+\theta_{x,y},\,xy).
\end{equation}
It is straightforward to check that this is a quasigroup operation, where (as we will need later)
\begin{align*}
    (a,x)\ldiv(b,y)&=(\psi_{x,x\ldiv y}^{-1}(b-\varphi_{x,x\ldiv y}(a)-\theta_{x,x\ldiv y}),\,x\ldiv y),\\
    (a,x)/    (b,y)&=(\varphi_{x/y,y}^{-1}(a-\psi_{x/y,y}(b)-\theta_{x/y,y}),\,x/y).
\end{align*}
The resulting quasigroup $(Q,\cdot,/,\ldiv)$ is called the \emph{abelian extension of $A$ by $F$ over a cocycle $\Gamma$}, and will be denoted by $Q=\abext{A}{F}{\Gamma}$. If $\Gamma$ is clear from the context or if it is not important, we speak of an \emph{abelian extension of $A$ by $F$}.

Let us briefly investigate the question ``When is $\abext{A}{F}{\Gamma}$ a loop?''

\begin{lemma}\label{Lm:abext_loop}
Let $A$ be an abelian group, $F$ a quasigroup, and $\Gamma=(\varphi,\psi,\theta)$ a cocycle. Then $Q=\abext{A}{F}{\Gamma}$ is a loop with neutral element $(a,x)$ if and only if $x=1$ is the neutral element of $F$ and
$$\varphi_{y,1}=\id = \psi_{1,y},\quad \varphi_{1,y}(a)+\theta_{1,y}=0=\psi_{y,1}(a)+\theta_{y,1}$$ hold for every $y\in F$.
\end{lemma}

\begin{proof}
Suppose that $Q$ is a loop with neutral element $(a,x)$. Then
\begin{align*}
    (0,y)&=(a,x)(0,y) = (\varphi_{x,y}(a)+\psi_{x,y}(0)+\theta_{x,y},xy) = (\varphi_{x,y}(a)+\theta_{x,y},xy),\\
    (0,y)&=(0,y)(a,x) = (\varphi_{y,x}(0)+\psi_{y,x}(a)+\theta_{y,x},yx) = (\psi_{y,x}(a)+\theta_{y,x},yx)
\end{align*}
for every $y\in F$. We deduce $xy=y=yx$, so $x=1$ is the neutral element of $F$, and also $\varphi_{1,y}(a)+\theta_{1,y}=0=\psi_{y,1}(a)+\theta_{y,1}$ for every $y\in F$. Using these facts, we have
\begin{align*}
    (b,y) &=(a,1)(b,y) = (\varphi_{1,y}(a)+\psi_{1,y}(b)+\theta_{1,y},y) = (\psi_{1,y}(b),y),\\
    (b,y) &=(b,y)(a,1) = (\varphi_{y,1}(b)+\psi_{y,1}(a)+\theta_{y,1},y) = (\varphi_{y,1}(b),y)
\end{align*}
for every $(b,y)\in Q$, and thus $\varphi_{y,1}=\id = \psi_{1,y}$ for every $y\in F$.

Conversely, it is straightforward to check that the stated cocycle conditions force $(a,1)$ to be the neutral element of $F$.
\end{proof}

In particular, the extension $Q=\abext{A}{F}{\Gamma}$ is a loop with neutral element $(0,1)$ if and only if $\varphi_{y,1}=\id = \psi_{1,y}$ and $\theta_{1,y} = 0 = \theta_{y,1}$ for every $y\in F$. As the following result shows, up to isomorphism, we can always assume that the neutral element of $Q=\abext{A}{F}{\Gamma}$ is $(0,1)$.

\begin{lemma}
Let $A$ be an abelian group, $F$ a loop with neutral element $1$, and $\Gamma=(\varphi,\psi,\theta)$ a cocycle. Suppose that $Q = \abext{A}{F}{\Gamma}$ has neutral element $(a,1)$. Then $\bar{\Gamma}=(\varphi,\psi,\bar{\theta})$ with $\bar{\theta}_{x,y} = \theta_{x,y}+\varphi_{x,y}(a) + \psi_{x,y}(a)-a$ is a cocycle, and $\bar{Q} = \abext{A}{F}{\bar{\Gamma}}$ is a loop with neutral element $(0,1)$ that is isomorphic to $Q$.
\end{lemma}

\begin{proof}
Note that $\bar{\Gamma}$ is a cocycle.
%Since $\abext{A}{F}{\Gamma}$ has neutral element $(a,1)$, the conditions of Lemma \ref{Lm:abext_loop} are satisfied. Then a straightforward calculation shows that the conditions of Corollary \ref{Cr:abext_loop} are satisfied by the cocycle $\bar{\Gamma}$, so $\bar{Q}$ is a loop with neutral element $(0,1)$. Another direct calculation shows that
The bijection $f:Q\to \bar{Q}$ defined by $f(b,y) = (b-a,y)$ is an isomorphism:
\begin{align*}
f(b,y)f(c,z)&=(\varphi_{y,z}(b-a)+\psi_{y,z}(c-a)+\bar\theta_{y,z},yz)\\
&=(\varphi_{y,z}(b)-\varphi_{y,z}(a)+\psi_{y,z}(c)-\psi_{y,z}(a)+\theta_{y,z}+\varphi_{y,z}(a) + \psi_{y,z}(a)-a,yz)\\
&=(\varphi_{y,z}(b)+\psi_{y,z}(c)+\theta_{y,z}-a,yz)=f((b,y)(c,z)).
\end{align*}
This isomorphism sends the neutral element $(a,1)$ of $Q$ onto $(0,1)$, which is in turn the neutral element of $\bar Q$.
\end{proof}

We are therefore led to the following definition. Let $A=(A,+,-,0)$ be a commutative group and $F=(F,\cdot,/,\ldiv,1)$ a loop. A triple $\Gamma=(\varphi,\psi,\theta)$ is called a \emph{loop cocycle} if $\varphi$, $\psi:F\times F\to\aut{A}$ and $\theta:F\times F\to A$ satisfy
\begin{displaymath}
    \varphi_{y,1}=\id = \psi_{1,y},\quad \theta_{1,y} = 0 = \theta_{y,1}
\end{displaymath}
for every $y\in F$.

For a loop cocycle $\Gamma=(\varphi,\psi,\theta)$ define multiplication on $A\times F$ by \eqref{Eq:abext}. The resulting algebra $\abext{A}{F}{\Gamma}$ is a loop with neutral element $(1,0)$, the \emph{abelian extension of $A$ by $F$ over a loop cocycle $\Gamma$}. This construction was originally presented in \cite{MTLBD} under the name ``affine quasidirect product.''

\begin{figure}
\begin{tikzpicture}[scale=0.07]
\draw (0,0)--(90,0);
\draw (0,0)--(0,-90);
\draw (0,-40)--(90,-40);
\draw (0,-80)--(90,-80);
\draw (40,0)--(40,-90);
\draw (80,0)--(80,-90);
\node at (20,5) {$(A,1)$};
\node at (60,5) {$(A,y)$};
\node at (-5,-20) [left] {$(A,1)$};
\node at (-5,-60) [left] {$(A,x)$};
\draw[<->] (45,-10) to node [right] {$\varphi_{1,y}$} (45,-35);
\draw[<->] (45,-50) to node [right] {$\varphi_{x,y}$} (45,-75);
\draw[<->] (10,-45) to node [below] {$\psi_{x,1}$} (35,-45);
\draw[<->] (50,-45) to node [below] {$\psi_{x,y}$} (75,-45);
\node at (70,-70) {$+\theta_{x,y}$};
\node at (85,-20) {$\cdots$};
\node at (85,-60) {$\cdots$};
\node at (20,-85) {$\vdots$};
\node at (60,-85) {$\vdots$};
\node at (85,-85) {$\ddots$};
\end{tikzpicture}
\caption{An abelian extension visualized.}
\label{Fig:abext}
\end{figure}
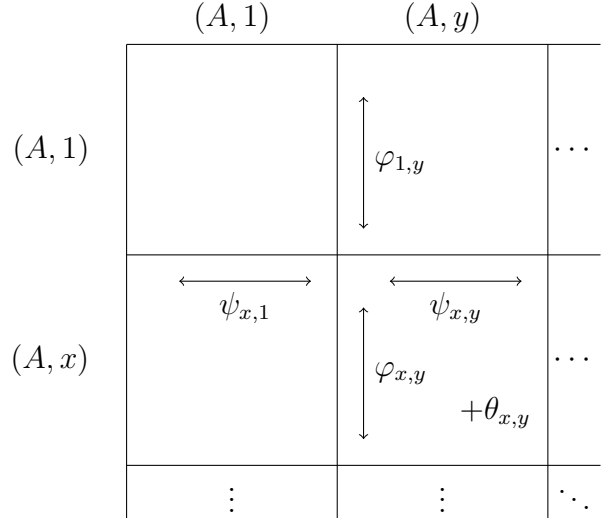

In Figure \ref{Fig:abext} we have visualized an abelian extension of $A$ by $F$, organizing the multiplication table of $Q=\abext{A}{F}{(\varphi,\psi,\theta)}$ according to the right cosets of $A$ in $Q$. The combinatorial nature of the construction is apparent from Figure \ref{Fig:abext}, given the fact that there are no conditions relating the constituents $\varphi$, $\psi$, $\theta$ of the cocycle.

\medskip

An abelian extension $\abext{A}{F}{(\varphi,\psi,\theta)}$ is called \emph{central} if $\varphi_{x,y}=\psi_{x,y}=\mathrm{id}$ for every $x$, $y\in F$, to be denoted shortly as $\varphi=\psi=1$. This is the classical concept of central extensions in loop theory that appears already in Bruck's fundamental work \cite{B}. Moreover, it is a special case of the more general central extensions in commutator theory for congruence modular varieties, see \cite[Chapter 7]{FM}.

\medskip

We conclude this section with several remarks on the structure of the multiplication group of abelian and central extensions of loops.

\begin{lemma}\label{Lm:form}
Let $Q=\abext{A}{F}{\Gamma}$ be an abelian extension of $A$ by $F$ over a loop cocycle $\Gamma$. Then every element $\gamma\in\mlt{Q}$ has the form
\begin{displaymath}
    \gamma(a,x) = (c_x+\gamma_x(a),C(x)),
\end{displaymath}
where $c_x\in A$, $\gamma_x\in\aut{A}$ and $C\in\mlt{F}$. Moreover,
\begin{itemize}
	\item[(i)] $\gamma\in\inn{Q}$ if and only if $c_1=0$ and $C\in\inn{F}$;
	\item[(ii)] if $Q=\abext{A}{F}{\Gamma}$ is a central extension, then $\gamma_x=\id$ for every $x\in F$.
\end{itemize}
\end{lemma}

\begin{proof}
Let $\gamma = L_{(b,y)}\in\mlt{Q}$. Then $\gamma(a,x) = (b,y)(a,x) = (\varphi_{y,x}(b)+\psi_{y,x}(a)+\theta_{y,x},yx)$, so $\gamma(a,x) = (c_x+\gamma_x(a),C(x))$ with $c_x=\varphi_{y,x}(b)+\theta_{y,x}\in A$, $\gamma_x=\psi_{y,x}\in\aut{A}$ and $C=L_y\in\mlt{F}$. Hence $L_{(b,y)}$ has the desired form. Similarly for $R_{(b,y)}$. If $\gamma$, $\delta\in\mlt{Q}$ are of the form $\gamma(a,x) = (c_x+\gamma_x(a),C(x))$, $\delta(a,x) = (d_x+\delta_x(a),D(x))$ with appropriate components, then
\begin{align*}
    \gamma^{-1}(a,x) &= (-\gamma^{-1}_{C^{-1}(x)}(c_{C^{-1}(x)}) + \gamma^{-1}_{C^{-1}(x)}(a),C^{-1}(x)),\\
    \gamma\delta(a,x) &= ((c_{D(x)}+\gamma_{D(x)}(d_x))+\gamma_{D(x)}\delta_x(a),CD(x))
\end{align*}
are of the desired form.

(i) Note that $\gamma\in\inn{Q}$ if and only if $(0,1) = \gamma(0,1) = (c_1+\gamma_1(0),C(1)) = (c_1,C(1))$. Equivalently, $c_1=0$ and $C\in\inn{F}$.

(ii) Using $\varphi=\psi=1$, we obtain $\gamma_x=\id$ for left and right translations, and this property propagates to inverses and compositions.
\end{proof}

\begin{proposition}\label{Pr:central_mlt}
Let $Q=\abext{A}{F}{\Gamma}$ be a central extension of $A$ by $F$ over a loop cocycle $\Gamma$.
\begin{itemize}
	\item[(i)] If $\mlt{F}$ is solvable of class $k$, then $\mlt{Q}$ is solvable of class at most $k+1$.
	\item[(ii)] If $\inn{F}$ is solvable of class $\ell$, then $\inn{Q}$ is solvable of class at most $\ell+1$.
\end{itemize}
\end{proposition}

\begin{proof}
By Lemma \ref{Lm:form}, any elements $\gamma$, $\delta\in\mlt{Q}$ are of the form $\gamma(a,x) = (a+c_x,C(x))$, $\delta(a,x) = (a+d_x,D(x))$, where $c_x$, $d_x\in A$ and $C$, $D\in\mlt{F}$. A quick calculation then shows that $$[\gamma,\delta](a,x) = \gamma\delta\gamma^{-1}\delta^{-1}(a,x)=(a+e_x,[C,D](x))$$
for some $e_x\in A$ (in fact, $e_x = -d_{D^{-1}(x)} - c_{C^{-1}D^{-1}(x)} + d_{C^{-1}D^{-1}(x)} + c_{DC^{-1}D^{-1}(x)}$). Since $\mlt{F}$ is solvable of class $k$, it follows by induction that all elements of $(\mlt{Q})^{(k)}$ are of the form $\gamma(a,x) = (a+c_x,x)$. Such mappings commute with one another, hence $\mlt{Q}$ is solvable of class at most $k+1$. Similarly for $\inn{Q}$.
\end{proof}

General abelian extensions are not as well behaved as central extensions with respect to solvability of $\mlt{Q}$ and $\inn{Q}$. The general expression for inner mappings, involving the automorphisms $\gamma_x$, indicates that $\inn Q$ may not be solvable even if $\inn F$ is, because the automorphism group of an abelian group is not necessarily solvable. The smallest such example is $A=\Z_2^3$ (with $\aut{A}$ isomorphic to the simple group $\mathrm{GL(3,2)}$) and $F=\Z_2$ (with $\inn{F}$ trivial). Indeed, computer calculations in the \texttt{LOOPS} package \cite{NaVo} for \texttt{GAP} \cite{GAP} show that there are many abelian extensions of $\Z_2^3$ by $\Z_2$ with non-solvable inner mapping groups. Nevertheless, the following problem is open:

\begin{problem}\label{Prob:inn_mlt_abelian}
Let $Q=\abext AF\Gamma$ be an abelian extension of $A$ by $F$. If both $\inn Q$ and $\mlt F$ are solvable, is $\mlt Q$ solvable?
\end{problem}

\section{Characterizing abelianess and centrality}\label{Sec:main}

We are now ready to state the main result, Theorem \ref{Th:main_abelian}, a characterization of abelian subloops.

\begin{theorem}\label{Th:main_abelian}
The following conditions are equivalent for a normal subloop $A$ of a loop $Q$:
\begin{enumerate}
\item[(A1)] $A$ is abelian in $Q$, that is, $[A,A]_Q=1$.
\item[(A2)] $W_{\bar u}W_{\bar{v}}^{-1}\restr{A}=\id$ for every inner word $W$ and every $u_i$, $v_i\in Q$ such that $u_i/v_i\in A$.
\item[(A3)] $\varphi\restr{A}\in\aut{A}$ for every $\varphi\in\inn{Q}$, and
$$[a,b]=[a,b,x]=[a,x,b]=[x,a,b]=1,\quad [a,x,u]=[a,x,v]$$ for every $a,b\in A$ and $x$, $u$, $v\in Q$ with $u/v\in A$.
\item[(A4)] $Q$ is an abelian extension of $A$ by $Q/A$.
\end{enumerate}
\end{theorem}

\begin{proof}
(A1) $\Rightarrow$ (A2): Let $u_i$, $v_i\in Q$ be such that $u_i/v_i\in A$. By Theorem \ref{Th:mainSV}, condition (A1) implies that $W_{\bar u}(a)=W_{\bar v}(a)$ for every $a\in A$ and every tot-inner word $W$. Hence $W_{\bar u}W_{\bar v}^{-1}\restr{A} = \id$ for every tot-inner word $W$. Condition (A2) follows as a special case.

(A2) $\Rightarrow$ (A3): Let $a$, $b\in A$, $x$, $u$, $v\in Q$ be such that $u/v\in A$. Upon using suitable inner words $W$ in (A2) we obtain
\begin{align*}
    [a,b]&=T_a(b)/b=T_1(b)/b=[1,b]=1,\\
    [a,b,x] &= R_{x,b}(a)/a = R_{x,1}(a)/a = [a,1,x] = 1,\\
    [a,x,b] &= R_{b,x}(a)/a = R_{1,x}(a)/a = [a,x,1] = 1,\\
    [a,x,u] &= R_{u,x}(a) = R_{v,x}(a) = [a,x,v].
\end{align*}
Moreover, $(xa)\ldiv(x\cdot ab) = L_{x,a}(b) = L_{x,1}(b) = b$, and multiplying both sides by $xa$, we obtain $[x,a,b]=1$. Since $A$ is preserved as a block by inner mappings, it remains to show that the restrictions of inner mappings to $A$ are homomorphisms. It is sufficient to check this for the generators $T_x$, $L_{x,y}$, $R_{x,y}$:
\begin{align*}
    T_x(a)T_x(b) &=T_x(a)((xb)/x) =(T_x(a)\cdot xb)/x                                    &\text{by Lemma \ref{Lm:useful}}\\
        &=(T_{xb}(a)\cdot xb)/x                                                             &\text{by (A2) with $W=T_x$}\\
        &=(xb\cdot a)/x = (x\cdot ba)/x                                                     &\text{by $[x,b,a]=1$}\\
        &=(x\cdot ab)/x                                                                     &\text{by $[a,b]=1$}\\
        &=T_x(ab),\\
    R_{x,y}(a)R_{x,y}(b) &= R_{x,y}(a)((by\cdot x)/(yx)) = (R_{x,y}(a)(by\cdot x))/(yx)     &\text{by Lemma \ref{Lm:useful}}\\
        &=(R_{x,by}(a)(by\cdot x))/(yx)                                                     &\text{by (A2) with $W=R_{x,y}$}\\
        &=((a\cdot by)x)/(yx)=((ab\cdot y)x)/(yx)                                           &\text{by $[a,b,y]=1$}\\
        &=R_{x,y}(ab),\\
    L_{x,y}(a)L_{x,y}(b) &= ((xy)\ldiv (x\cdot ya))L_{x,y}(b) = (xy)\ldiv ((x\cdot ya)L_{x,y}(b))    &\text{by Lemma \ref{Lm:useful}}\\
        &= (xy)\ldiv ((x\cdot ya)L_{x,ya}(b))                                               &\text{by (A2) with $W=L_{x,y}$}\\
        &= (xy)\ldiv (x(ya\cdot b)) = (xy)\ldiv (x(y\cdot ab))                              &\text{by $[y,a,b]=1$}\\
        &=L_{x,y}(ab).
\end{align*}

(A3) $\Rightarrow$ (A4): We see immediately that $A$ is an abelian group, and we will retain the multiplicative notation for $A$ here. Let $F$ be a transversal to $A$ in $Q$ containing $1$. Define new multiplication $\circ$ on $F$ by letting $x\circ y$ be the unique element of $F$ such that $A(x\circ y) = A(xy)$. Consider the bijection $g:F\to Q/A$, $x\mapsto Ax$. Then $g(x\circ y) = A(x\circ y) = A(xy) = Ax\cdot Ay = g(x)g(y)$, so $(F,\circ)$ is a loop isomorphic to $Q/A$. Every element of $Q$ has a unique decomposition $ax$ with $a\in A$ and $x\in F$, hence the mapping $f:A\times F\to Q$, $(a,x)\mapsto ax$ is a bijection. We will find a loop cocycle $\Gamma = (\varphi,\psi,\theta)$ such that $f$ becomes an isomorphism $\abext{A}{(F,\circ)}{\Gamma}\to Q$.
For every $x$, $y\in Q$ let
$$\varphi_{x,y}=R_{y,x}\restr{A},\quad \psi_{x,y}= (R_{xy}^{-1}L_xR_y)\restr{A},\quad \theta_{x,y} = (xy)/(x\circ y).$$
We clearly have $\theta_{x,y}\in A$, and, by (A3), $\varphi_{x,y}$, $\psi_{x,y}\in\aut{A}$. Note that $\varphi_{x,1}=\psi_{1,x}=\id$ and $\theta_{x,1}=\theta_{1,x}=1$, so $\Gamma$ is a loop cocycle. Our goal is to establish the equality of
\begin{displaymath}
    f(a,x)f(b,y)=ax\cdot by\quad\text{and}\quad f((a,x)(b,y)) = (\varphi_{x,y}(a)\psi_{x,y}(b)\theta_{x,y})\cdot (x\circ y)
\end{displaymath}
for every $a$, $b\in A$ and $x$, $y\in F$. Note that no additional parentheses are needed in the last expression since $A$ is associative. Now,
\begin{align*}
    ax\cdot by&= ((ax\cdot by)/(x\cdot by))\cdot (x\cdot by) = ([a,x,by]a)(x\cdot by)\\
        & = ([a,x,y]a)(x\cdot by)                                      &\text{by ``$[a,x,u] = [a,x,v]$''}\\
        & = \varphi_{x,y}(a)(x\cdot by) = \varphi_{x,y}(a)\cdot ((x\cdot by)/(xy))(xy)\\
        & = \varphi_{x,y}(a)\cdot(\psi_{x,y}(b)(xy)) = (\varphi_{x,y}(a)\psi_{x,y}(b))\cdot(xy).    &\text{by ``$[a,b,x] = 1$''.}
\end{align*}
Therefore
\begin{align*}
    (ax\cdot by)/(x\circ y) &= (\varphi_{x,y}(a)\psi_{x,y}(b)\cdot (xy))/(x\circ y) \\
                            &= \varphi_{x,y}(a)\psi_{x,y}(b)\cdot ((xy)/(x\circ y)) &\text{by Lemma \ref{Lm:useful}}\\
                            &= \varphi_{x,y}(a)\psi_{x,y}(b)\theta_{x,y},
\end{align*}
as desired.

(A4) $\Rightarrow$ (A1) Let $Q=\abext{A}{F}{\Gamma}$, where $\Gamma = (\varphi,\psi,\theta)$ is a loop cocycle. Note that $(a,x)/(b,y)\in A\times 1$ if and only if $x=y$, hence it is sufficient to show that $W_{(a_1,x_1),\dots,(a_n,x_n)}(c,1)$ is independent of $a_1$, $\dots$, $a_n$ for every tot-inner word $W$, every $(a_i,x_i)\in Q$ and every $c\in A$. In fact, it suffices to check this condition for the standard generators of $\tinn{Q}$. A lengthy but straightforward calculation yields
\begin{align*}
    T_{(a,x)}(c,1)&=(a+\psi_{x,1}(c),x)\,/\,(a,x)=(\varphi_{1,x}^{-1}\psi_{x,1}(c),1),\\
    L_{(a,x),(b,y)}(c,1)&=(\varphi_{x,y}(a)+\psi_{x,y}(b)+\theta_{x,y},\,xy)\,\ldiv\,(\varphi_{x,y}(a)+\psi_{x,y}(b+\psi_{y,1}(c))+\theta_{x,y},\,xy)\\
        &=(\psi_{xy,1}^{-1}\psi_{x,y}\psi_{y,1}(c),1),\\
    R_{(a,x),(b,y)}(c,1)&=(\varphi_{y,x}(\varphi_{1,y}(c)+b) + \psi_{y,x}(a) + \theta_{y,x},\,yx)/(\varphi_{y,x}(b)+\psi_{y,x}(a)+\theta_{y,x},\,yx)\\
        &=(\varphi_{1,yx}^{-1}\varphi_{y,x}\varphi_{1,y}(c),\,1),\\
    U_{(a,x)}(c,1)&=(a-\varphi_{1,x}(c),x)\,/\,(a,x)=(-c,1),\\
    M_{(a,x),(b,y)}(c,1)&=(\psi_{y,y\ldiv x}^{-1}(a-\varphi_{y,y\ldiv x}(b)-\theta_{y,y\ldiv x}),\,y\ldiv x)\\
        &\phantom{===\ }/(\psi_{y,y\ldiv x}^{-1}(a-\varphi_{y,y\ldiv x}(b-\varphi_{1,y}(c))-\theta_{y,y\ldiv x}),\,y\ldiv x)\\
        &=(-\varphi_{1,y\ldiv x}^{-1}\psi_{y,y\ldiv x}^{-1}\varphi_{y,y\ldiv x}\varphi_{1,y}(c),1).
\end{align*}
\end{proof}

Using Theorem \ref{Th:main_abelian}, we can easily deduce the following well-known characterization of central subloops, formatted in an analogous way in order to highlight the similarities between the two results.

\begin{theorem}\label{Th:main_central}
The following conditions are equivalent for a normal subloop $A$ of a loop $Q$:
\begin{enumerate}
\item[(C1)] $A$ is central in $Q$, that is, $[A,Q]_Q=1$.
\item[(C2)] $W_{\bar{u}}\restr{A}=\id$ for every inner word $W$ and every $u_i\in Q$.
\item[(C3)] $\varphi\restr{A}=\id$ for every $\varphi\in\inn{Q}$.
\item[(C3')] $[a,x]=[a,x,y]=[x,a,y]=[x,y,a]=1$ for every $a\in A$, $x$, $y\in Q$.
\item[(C4)] $Q$ is a central extension of $A$ by $Q/A$.
\end{enumerate}
\end{theorem}

\begin{proof}
(C1) $\Rightarrow$ (C2) follows from Theorem \ref{Th:mainSV} and from the observation that with $u_i=1$ we always have $W_{\bar u}(a)=a$.
To show (C2) $\Rightarrow$ (C3), observe that every $\varphi\in\inn{Q}$ can be expressed as a word in the standard generators $T_x$, $L_{x,y}$, $R_{x,y}$.

Next, we prove that (C3) $\Leftrightarrow$ (C3').
Let $a\in A$, $x$, $y\in Q$. We have
\begin{align*}
T_x(a)=a \quad&\text{iff}\quad T_x(a)/a=[x,a]=1 \quad\text{iff}\quad [a,x]=1,\\
R_{y,x}(a)=a \quad&\text{iff}\quad R_{y,x}(a)/a = [a,x,y] = 1,\\
L_{x,y}(a)= (xy)\ldiv(x\cdot ya) = a \quad&\text{iff}\quad x\cdot ya = xy\cdot a \quad\text{iff}\quad [x,y,a]=1.
\end{align*}
This proves (C3') $\Rightarrow$ (C3), since $T_x$, $R_{x,y}$ and $L_{x,y}$ generate $\inn{Q}$. To finish (C3) $\Rightarrow$ (C3'), we proceed as usual and obtain $(xa)y = (ax)y = a(xy) = (xy)a = x(ya) = x(ay)$, that is, $[x,a,y]=1$.

Now, if both (C3) and (C3') hold, condition (A3) of Theorem \ref{Th:main_abelian} is satisfied, too, and we can follow the proof (A3) $\Rightarrow$ (A4) therein. Since the inner mappings $\varphi_{x,y}$, $\psi_{x,y}$ are now identical on $A$, we have established (C4).

Finally, if (C4) holds, the calculation in the proof of (A4) $\Rightarrow$ (A1) combined with the fact that $\varphi=\psi=1$ imply that the mappings $T_x$, $L_{x,y}$, $R_{x,y}$ are identical on $A$, and the mappings $U_x$ and $M_{x,y}$ satisfy $U_x(a)=M_{x,y}(a)=-a$ for every $a\in A$. Thus $W_{\bar u}(a)=W_{\bar v}(a)$ for every tot-inner word $W$ and every $a\in A$, $u_i$, $v_i\in Q$, and (C1) follows from Theorem \ref{Th:mainSV}.
\end{proof}

We conclude this section with several remarks on the statement and proof of Theorems \ref{Th:main_abelian} and \ref{Th:main_central}.

\subsubsection*{Optimality of the syntactic conditions}

In the proof of (A3) $\Rightarrow$ (A4) we seemingly did not use the fact that $\varphi\restr{A}\in\aut{A}$ for every $\varphi\in\inn{Q}$, only that $\varphi_{x,y}=R_{y,x}$, $\psi_{x,y}= R_{xy}^{-1}L_xR_y$ have this property. But since $\inn{Q} = \langle \varphi_{x,y},\psi_{x,y}:x,y\in Q\rangle$ by \cite[Proposition 3.2]{SV}, we did in fact use the assumption in full.

The identity $[x,a,y]=1$ can be removed from condition (C3'), as follows from the proof of Theorem \ref{Th:main_central}. (In fact, any one of the three associator identities can be removed from (C3').) Note that condition (C3') describes equationally the fact that $A\leq Z(Q)$.

We failed to find a compact, purely associator-commutator characterization of abelian normal subloops in the spirit of condition (C3'), but the following discussion of (A3) suggests that no such characterization might exist. Let us label the conditions in (A3) as follows: i) $\varphi\restr{A}\in\aut{A}$ for every $\varphi\in\inn{Q}$, ii) $[a,b]=1$, iii) $[a,b,x]=1$, iv) $[a,x,b]=1$, v) $[x,a,b]=1$, vi) $[a,x,u]=[a,x,v]$, with variables quantified as in (A3).

The condition iv) is an obvious consequence of vi) and can therefore be removed from (A3).

The condition i) by itself is rather weak. Recall that a loop $Q$ is said to be \emph{automorphic} if $\inn{Q}\le\aut{Q}$, see \cite{BP} and \cite{KKPV}. There exist (commutative) automorphic loops that are not associative. Taking such a loop $Q$ and letting $A=Q$, we see that i) does not imply any of the remaining conditions, in fact, even i) with ii) do not imply any of the remaining conditions.

On the other hand, condition i) cannot be removed from (A3). As in \cite[Section 9]{SV}, for an abelian group $(G,+)$ and a quasigroup $(G,\oplus)$ define $Q=G[\oplus]$ to be the loop on $G\times \Z_2$ with multiplication
\begin{displaymath}
    (x,a)(y,b)=\left\{
	\begin{array}{ll}
		(x+y,a+b)  & \text{if $a=0$ or $b=0$,}\\
		(x\oplus y,0) & \text{otherwise.}
	\end{array}\right.
\end{displaymath}
With $G=\Z_2^2$, there is a loop $Q=G[\oplus]$ and a normal subloop $A=G\times 0$ of $Q$ that satisfies ii)--vi) but in which i) fails.

Any non-commutative group $Q$ with $A=Q$ satisfies i), iii)--vi) but not ii), so ii) cannot be removed from (A3). There is a loop $Q$ of the form $\mathbb Z_4[\oplus]$ that satisfies i)--v) but not vi), so vi) cannot be removed from (A3).

We do not know whether iii) or v) can be removed from (A3). However, there is an $8$-element example with $A=\mathbb Z_4$ that satisfies ii), iv), v), vi) but not iii), and another such example that satisfies ii), iii), iv), vi) but not v). 

\subsubsection*{Inner mappings, or tot-inner mappings?}

As we mentioned earlier, while developing the commutator theory for loops, it proved useful to work with tot-inner mappings and tot-inner words.
To date, we do not know whether Theorem \ref{Th:mainSV} remains true if ``tot-inner'' and ``$\tinn Q$'' are replaced by ``inner'' and ``$\inn Q$'' in the condition imposed on the set $\W$ that is used for generating the commutator $[A,B]_Q$. Theorems \ref{Th:main_abelian} and \ref{Th:main_central} show that it is sufficient to consider inner mappings for certain special types of commutators, namely $[A,A]_Q$ and $[A,Q]_Q$. (This is the equivalence of conditions (A1), (A2) and (C1), (C2), respectively.)

It turns out that for a normal subloop $A$ that is abelian in $Q$ we have $\varphi\restr{A}\in\aut{A}$ for every tot-inner mapping $\varphi$, too, not just for inner mappings. Indeed, suppose that the equivalent conditions of Theorem \ref{Th:main_abelian} hold, and let $a$, $b\in A$ and $x$, $y\in Q$. First note that the condition $[a,b]=[a,b,x]=[a,x,b]=[x,a,b]=1$ implies $(ab)\ldiv x=a\ldiv(b\ldiv x)=b\ldiv(a\ldiv x)$, which we will use freely. Then
\begin{align*}
    U_x(a)U_x(b) &=U_x(a)\cdot((b\ldiv x)/x) =(U_x(a)\cdot(b\ldiv x))/x                 &\text{by Lemma \ref{Lm:useful}}\\
        &=(U_{b\ldiv x}(a)\cdot(b\ldiv x))/x                                            &\text{by (A1) with $W=U_x$}\\
        &=(a\ldiv(b\ldiv x))/x = ((ab)\ldiv x)/x=U_x(ab),  \\
    M_{x,y}(a)M_{x,y}(b)&=M_{x,y}(a)M_{x,a\ldiv y}(b)                                   &\text{by (A1) with $W=M_{x,y}$}\\
        &=M_{x,y}(a)\cdot(((a\ldiv y)\ldiv x)/((b\ldiv(a\ldiv y))\ldiv x))\\
        &=(M_{x,y}(a)((a\ldiv y)\ldiv x))/((b\ldiv(a\ldiv y))\ldiv x)              &\text{by Lemma \ref{Lm:useful}}\\
        &=(y\ldiv x)/((b\ldiv (a\ldiv y))\ldiv x)\\
        &=(y\ldiv x)/((ab\ldiv y)\ldiv x)=M_{x,y}(ab).
\end{align*}

\subsubsection*{Choosing associators and commutators}

The particular form of the associator $[x,y,z]$ comes into play only in the statement of condition (A3), namely in the identity $[a,x,u]=[a,x,v]$. In all other instances the identities are of the form $[x,y,z]=1$ with various constraints on the variables, which have the same meaning for all reasonably defined associators.

We could have chosen the associator $[x,y,z]$ arbitrarily as long as we have maintained the requirement that it is a deviation of an inner mapping from the identity mapping---this is crucial to obtain the automorphic property of $\varphi_{x,y}$ in the proof of (A3) $\Rightarrow$ (A4).

The particular form of the commutator $[x,y]$ never comes into play in the statement of the theorems.

\subsubsection*{On the loop cocycle condition}

The mapping $\varphi_{x,y} = R_{y,x}$ in the implication (A3) $\Rightarrow$ (A4) satisfies not only $\varphi_{y,1}=\id$ but also $\varphi_{1,y}=\id$. We could therefore strengthen the definition of a loop cocycle by demanding also $\varphi_{1,y}=\id$. We decided not to do this for symmetry reasons, and also to match our definition with that of \cite{MTLBD}.

\begin{problem}
Does Theorem \ref{Th:mainSV} remain true if ``tot-inner'' and ``$\tinn Q$'' are replaced by ``inner'' and ``$\inn Q$'' in the condition imposed on the set $\W$ that is used for generating the commutator $[A,B]_Q$?
\end{problem}

\section{Solvability and nilpotence}\label{Sec:sn}

Call a loop $Q$ an \emph{iterated abelian extension} (resp. \emph{iterated central extension}) if it is either an abelian group, or it is an abelian extension (resp. central extension) $Q = \abext{A}{F}{\Gamma}$ of an abelian group $A$ by an iterated abelian extension (resp. iterated central extension) $F$. Thanks to the fact that abelian extensions of loops are precisely the affine quasidirect products of \cite{MTLBD}, our iterated abelian extensions are precisely the \emph{polyabelian loops} of \cite{MTLBD}.

The definitions of solvability and nilpotence can be rephrased in the following way. A loop $Q$ is \emph{congruence solvable} (resp. \emph{centrally nilpotent}) if it is either an abelian group, or if it contains a normal subloop $A$ that is abelian in $Q$ (resp. central in $Q$) and $Q/A$ is congruence solvable (resp. centrally nilpotent). As an immediate consequence of Theorems \ref{Th:main_abelian} and \ref{Th:main_central} we obtain:

\begin{corollary}\label{Cr:iterated_abelian}
A loop is congruence solvable if and only if it is an iterated abelian extension.
\end{corollary}

\begin{corollary}\label{Cr:iterated_central}
A loop is centrally nilpotent if and only if it is an iterated central extension.
\end{corollary}

Recall once again that the traditional notion of solvability in loop theory, which we call classical solvability here, is weaker than congruence solvability. On the other hand, a stronger concept of nilpotence, not directly related to centrality, has recently emerged in universal algebra. In \cite{Bu}, Bulatov introduced higher commutators and the related notion of nilpotence, called \emph{supernilpotence}. A theory of higher commutators for algebras with a Mal'tsev term was developed by Aichinger and Mudrinski \cite{AM}. We will not go into details here but allow us to state the following result of \cite{AM} that applies to finite loops: \emph{A finite algebra with a Mal'tsev term is supernilpotent if and only if it is a direct product of centrally nilpotent algebras of prime power size.}

One can view the four notions (classical solvability, congruence solvability, central nilpotence and supernilpotence), together with their class, as a hierarchy of properties, with \emph{abelianess} (defined as nilpotence of class $1$ or solvability of class $1$) at the top. Indeed, a loop is abelian in this sense if and only if it is an abelian group.

Many results in loop theory relate solvability and nilpotence of loops to the analogous properties of their associated permutation groups, the multiplication group and the inner mapping group. We conclude the paper with a summary of known results in this area, and we point out how the new concepts of congruence solvability and supernilpotence fit into the hierarchy.

\begin{figure}
\begin{tikzpicture}[scale=0.1]
\node at (0,10) {$Q$};
\node at (40,10) {$\mlt{Q}$};
\node at (80,10) {$\inn{Q}$};
\draw[fill] (0,0) circle [radius=1];        % Q abelian
\draw[fill] (40,0) circle [radius=1];       % Mlt(Q) abelian
\draw[fill] (80,-10) circle [radius=1];     % Inn(Q) abelian
\draw[dotted] (-35,-15)--(100,-15);         % abelian-nilpotent separator
\draw[fill] (0,-25) circle [radius=1];      % Q supernilpotent
\draw[fill] (40,-25) circle [radius=1];     % Mlt(Q) nilpotent
\draw[fill] (80,-35) circle [radius=1];     % Inn(Q) nilpotent
\draw[dotted] (-35,-35)--(10,-35);          % supernilpotent-nilpotent half separator
\draw[fill] (0,-45) circle [radius=1];      % Q centrally nilpotent
\draw[dotted] (-35,-55)--(100,-55);         % nilpotent-solvable separator
\draw[fill] (0,-60) circle [radius=1];      % Q congruence solvable
\draw[fill] (40,-60) circle [radius=1];     % Mlt(Q) solvable
\draw[dotted] (-35,-65)--(10,-65);          % congruence solvable - classically solvable half separator
\draw[fill] (0,-75) circle [radius=1];      % Q classically solvable
\draw[fill] (80,-75) circle [radius=1];     % Inn(Q) solvable
\draw[dotted] (10,5)--(10,-80);             % loop concept - group concept separator
\draw (0,0)--(40,0)--(80,-10);              % abelian layer
\draw (40,0)--(40,-25);                     % Mlt(Q) abelian - Mlt(Q) nilpotent
\draw (80,-10)--(80,-35);                   % Inn(Q) abelian - Inn(Q) nilpotent
\draw (40,-25)--(80,-35);                   % Mlt(Q) nilpotent - Inn(Q) nilpotent
\draw[dashed] (0,-25) to node [above] {\emph{Wright}} (40,-25);             % Wright
\draw (40,-25) to node [sloped,above] {\emph{Bruck}} (0,-45);               % Bruck
\draw[dashed] (80,-35) to node [sloped,above] {\emph{Niemenmaa}} (0,-45);   % Niemenmaa
\draw (0,0)--(0,-75);                       % loop implications
\draw (0,-45) to node [sloped,above] {\emph{Bruck}} (40,-60);               % Bruck
\draw[dashed] (40,-60) to node [sloped,above] {\emph{Vesanen}} (0,-75);     % Vesanen
\draw (40,-60)--(80,-75);                   % Mlt(Q) solvable - Inn(Q) solvable
\draw (80,-35)--(80,-75);                   % Inn(Q) nilpotent - Inn(Q) solvable
\node at (-5,-5) [left] {abelian};
\node at (-5,-25) [left] {supernilpotent};
\node at (-5,-45) [left] {centrally nilpotent};
\node at (-5,-60) [left] {congruence solvable};
\node at (-5,-75) [left] {classically solvable};
\node at (85,-5) [right] {abelian};
\node at (85,-35) [right] {nilpotent};
\node at (85,-65) [right] {solvable};
\end{tikzpicture}
\caption{Abelianess, nilpotence and solvability of loops and the associated permutation groups.}
\label{Fig:bigpic}
\end{figure}
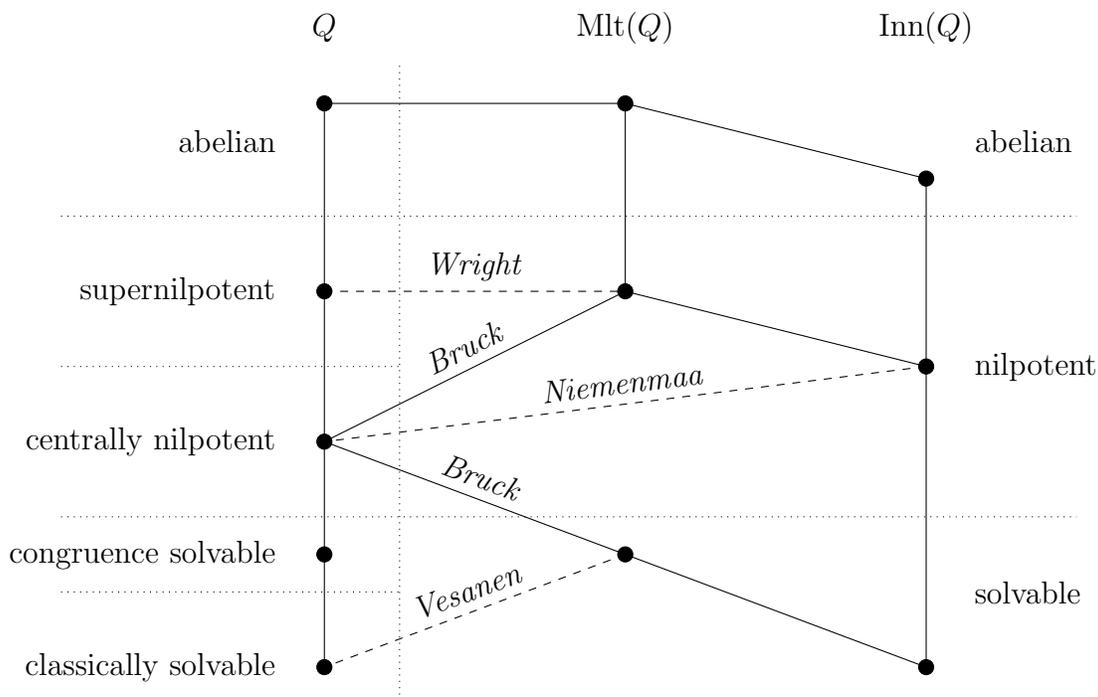

The known implications among these properties are summarized in Figure \ref{Fig:bigpic}. A horizontal edge means equivalence, slanted or vertical edges designate implications, with stronger notions higher up. A dashed edge means that a proof is known only for finite loops. Labeled edges are nontrivial theorems, the remaining edges are easy facts. We are not aware of infinite counterexamples to the theorems that assume finiteness.
The references for Figure \ref{Fig:bigpic} are:
\begin{itemize}
	\item $\mlt Q$ nilpotent $\Rightarrow$ $Q$ centrally nilpotent: \cite[Section 8]{Br46},
	\item $Q$ centrally nilpotent $\Rightarrow$ $\mlt Q$ solvable: \cite[Section 8]{Br46},
	\item for $Q$ finite, $\mlt Q$ nilpotent $\Leftrightarrow$ $Q$ supernilpotent: \cite{W},
	\item for $Q$ finite, $\inn Q$ nilpotent $\Rightarrow$ $Q$ centrally nilpotent: \cite{N}, building on \cite{Ma},
	\item for $Q$ finite, $\mlt Q$ solvable $\Rightarrow$ $Q$ classically solvable: \cite{V}.
\end{itemize}
And here is some information about the counterexamples in Figure \ref{Fig:bigpic}:
\begin{itemize}
	\item $Q$ centrally nilpotent $\not\Rightarrow$ $Q$ supernilpotent or $\mlt Q$ nilpotent: any nilpotent loop that is not a direct product of loops of prime power order (there is one of order $6$),
	\item $Q$ centrally nilpotent $\not\Rightarrow$ $\inn Q$ nilpotent: certain loop of order 16 and nilpotence class 3,
	\item $\mlt Q$ solvable $\not\Rightarrow$ $Q$ congruence solvable: certain non-abelian extension of $\Z_4$ by $\Z_2$, see \cite{SV},
	\item $Q$ congruence solvable $\not\Rightarrow$ $\inn Q$ solvable: certain abelian extension of $\Z_2^3$ by $\Z_2$, see \cite{MTLBD} or Section \ref{Sec:extensions}.
\end{itemize}
Among the open problems related to Figure \ref{Fig:bigpic}, we would like to state explicitly one that has recently received considerable attention:

\begin{conjecture}
Let $Q$ be a loop. If $\inn Q$ is abelian, then $Q$ is centrally nilpotent of class at most $3$.
\end{conjecture}

A finer version of the conjecture is discussed in \cite{KVV}. Existence of loops (necessarily not associative) $Q$ of nilpotence class $3$ with abelian $\inn{Q}$ was first shown by Cs\"org\H{o} \cite{Cs}. See also \cite{DrVo} for a large class of examples.

\medskip

We finish with an observation that relates our results to a particular edge in Figure \ref{Fig:bigpic}. Recall the classical result of Bruck \cite[Corollary II after Theorem 8B]{Br46}: \emph{A centrally nilpotent loop $Q$ has solvable $\mlt Q$}. Note that this result follows immediately from our Proposition \ref{Pr:central_mlt} and Corollary \ref{Cr:iterated_central}. The theorem does not extend to congruence solvability, but the following problem is open:

\begin{problem}\label{Prob:inn_mlt_solvable}
Let $Q$ be a congruence solvable loop with $\inn{Q}$ solvable. Is $\mlt{Q}$ solvable?
\end{problem}

\begin{proposition}
If the answer to Problem \ref{Prob:inn_mlt_abelian} is positive than the answer to Problem \ref{Prob:inn_mlt_solvable} is also positive.
\end{proposition}

\begin{proof}
Suppose that $\inn{Q}$ is solvable. We proceed by induction on the length of a shortest iterated abelian extension for $Q$. If $Q$ is an abelian group, then $\mlt{Q}\cong Q$ is solvable. Let $Q=\abext{A}{F}{\Gamma}$ be an abelian extension such that $F$ is congruence solvable and with a shorter iterated abelian extension than $Q$. A result by Albert (see \cite[Lemma 8A]{Br46}, for instance) says that for every loop $L$ and its normal subloop $N$ the group $\inn{L/N}$ is isomorphic to a quotient of $\inn{L}$; in fact, $$\inn{L/N}\cong \inn L/\{\varphi\in\inn L:\varphi(x)\in xN\text{ for every }x\in L\}.$$
Since $\inn{Q}$ is solvable, we deduce that $\inn{F}\cong\inn{Q/A}$ is solvable. The induction assumption then applies to $F$, so $\mlt F$ is solvable. Finally, the positive answer to Problem \ref{Prob:inn_mlt_abelian} then ensures that $\mlt Q$ is solvable.
\end{proof}

\section*{Acknowledgement}

We thank Michael Kinyon for many useful discussions about the topic of this paper, particularly about Proposition \ref{Pr:central_mlt}.

\end{document}